\documentclass{amsart}
\usepackage[final]{hyperref}
\usepackage{geometry}
\title{On a Differential Model for Sandpiles 
Growing in a Silo} 

\author{Graziano Crasta}
\address{Dipartimento di Matematica ``G.\ Castelnuovo'', Sapienza Università di Roma\\
	P.le A.\ Moro 5, I-00185 Roma (Italy)}
\email{graziano.crasta@uniroma1.it}
\author{Annalisa Malusa}
\address{Dipartimento di Matematica ``G.\ Castelnuovo'', Sapienza Università di Roma\\
	P.le A.\ Moro 5, I-00185 Roma (Italy)}
\email{annalisa.malusa@uniroma1.it}

\theoremstyle{plain}
\newtheorem{theorem}{Theorem}
\newtheorem{lemma}{Lemma}
\newtheorem{proposition}{Proposition}

\newtheorem{conjecture}{Conjecture}
\theoremstyle{definition}
\newtheorem{definition}{Definition}
\newtheorem{remark}{Remark}
\newtheorem{example}{Example}

\numberwithin{equation}{section}
\numberwithin{theorem}{section}
\numberwithin{lemma}{section}
\numberwithin{proposition}{section}
\numberwithin{corollary}{section}
\numberwithin{conjecture}{section}
\numberwithin{definition}{section}
\numberwithin{remark}{section}
\numberwithin{example}{section}

\def\R{\mathbb{R}}
\newcommand{\Xp}[1][\phi]{\mathbb{X}_{#1}}
\def\Xf{\Xp[f]}
\def\meas{\mathcal{M}}

\def\cray#1{{[\![#1]\!]}}
\DeclareMathOperator{\dive}{div}
\DeclareMathOperator{\spt}{supp}
\DeclareMathOperator{\diam}{diam}
\def\uhl{u_\phi}
\def\ul{u_\infty}
\def\uu{\underline{u}}
\def\fr{\widetilde{f}}

\DeclareMathOperator{\lip}{Lip}
\def\lipu{\lip_1}

\begin{document}
\begin{abstract}
We discuss some features of a boundary value problem for a system of PDEs that describes the growth of a sandpile in a container under the action of a vertical source. In particular, we characterize the long--term behavior of the profiles, and we provide a sufficient condition on the vertical source that guarantees the convergence to the equilibrium in a finite time. We show by counterexamples that a stable configuration may not be reached in a finite time, in general, even if the source is time-independent. 
Finally, we provide a complete characterization of the equilibrium profiles.
\end{abstract}
\maketitle
\tableofcontents

\section{Introduction}

Since the pioneering paper \cite{Pr}, the variational approach to the study of growing sandpiles 
has become established as an effective way to describe the macroscopic behavior of granular materials. 
In these models, the complex dynamics of granular flow is simplified  by dividing  the material into a static lower layer (standing layer), which contains most of the pile, and a flowing, dynamic upper layer (rolling layer). 
This approach is particularly effective for simulating the evolution of a sandpile as new material is added 
(see \cite{CCS,DPJ}) and its equilibrium shape (see \cite{CCCG,CFV,CM6,CM7,CM9,CM8,CM10,CM11}). 

We are interested in the evolution of a sandpile growing in a bounded container under the action of a vertical source. 
The container has a flat base $\Omega\subset\R^2$ and vertical walls, whose height is specified by a function 
$\phi \colon \partial \Omega \to [0,+\infty[$. 
The vertical source, which is assumed time--independent, is modeled here by a function $f\colon \Omega \to [0,+\infty[$. 
At each time $t\geq 0$,
the shape of the sandpile (that is, the profile of the standing layer)
is described by the graph of a function $u(t, \cdot)$, with $u\colon \R^+\times \Omega \to \R$. 
We denote by $u_0$ the initial profile of the pile. 
A key feature of the granular matter is the existence of a critical slope which cannot be exceeded by the standing layer. 
In the following we normalize the critical slope to $1$, and hence we impose the constraint $|\nabla u|\leq 1$
on the spatial gradient of $u$. 
The thickness of the rolling layer is given by $v\colon \R^+\times \Omega \to [0,+\infty[$, and it is assumed that  the material emitted from the source rolls downhill only if it falls on points with a critical slope, that is 
$(1-|\nabla u|)v=0$ in $\R^+ \times \Omega$. 
When the profile reaches the top of the wall (that is, at those points of $\partial \Omega$ where $u=\phi$), 
the sand coming from the rolling layer falls. 
We therefore introduce a third variable in our problem: 
a non-negative measure $\nu$ supported on $\partial \Omega$ and describing the amounts of sand discharged at each point. 
Assuming that the material rolls along the directions of steepest descent, the mass conservation law can be written as
$
\partial_t u - \dive(v\, \nabla u) = f - \nu.
$

In summary, given a silo $(\Omega,\phi)$ and a vertical source $f$,  
the dynamics of the corresponding growing sandpile is described by a triplet $(u,v,\nu)$ satisfying 
the following system of PDEs with constraints: 
\begin{equation}
	\label{f:Pintro}
	\begin{cases}
		\partial_t u - \dive(v\, \nabla u) = f - \nu
		&\text{in}\ \R^+\times\R^N,
		\\
		|\nabla u|\leq 1,\ v\geq 0, \ (1-|\nabla u|) v = 0,
		&\text{in}\ \R^+\times\Omega,
		\\
		0\leq u(t, x) \leq \phi(x)
		&\forall t\geq 0,\ x\in\partial\Omega,
		\\
		u(t,x) = \phi(x)
		& \text{$\nu$-a.e.\  $(t,x)\in\R^+\times\partial\Omega$},
		\\
		u(0, \cdot) = u_0.
	\end{cases}
\end{equation}
Since the analysis can be done in any space dimension $N\geq 1$, 
in the above problem and in the following we will assume that $\Omega$ is an open bounded convex subset
of $\R^N$.
Moreover, to simplify the exposition, we will assume that $u_0 = 0$, i.e., we start the evolution with an empty silo.
Yet, we will consider a possibly non-vanishing initial data in the analysis of 
a related variational inequality for the $u$-component (see Proposition~\ref{p:comp}),
which in turn will be useful for the study of stationary solutions of~\eqref{f:P},
i.e., the solutions $(u,v,\nu)$ 
of
\begin{equation*}
	\begin{cases}
		- \dive(v\, \nabla u) = f - \nu
		&\text{in}\ \R^N,
		\\
		|\nabla u| \leq 1,\ v\geq 0, \ (1-|\nabla u|) v = 0,
		&\text{a.e.\ in}\ \Omega,
		\\
		0\leq u \leq \phi
		&\text{in}\ \partial\Omega,
		\\
		u(x) = \phi(x)
		& \text{$\nu$-a.e.\  $x\in\partial\Omega$}
	\end{cases}
\end{equation*}
(see Section~\ref{s:tre}).
The above problem can be formulated also without any explicit reference to the boundary measure $\nu$,
but referring only to its support $\Gamma_f$ that can be explicitly constructed (see \eqref{f:Gf}).

\bigskip
The plan of the paper is the following.

In Section \ref{s:uno} we state the hypotheses on $\Omega$, $\phi$, and $f$ that guarantee the existence of a solution of (the $N$--dimensional weak formulation of) \eqref{f:Pintro}, as proved in \cite{DPJ}. 

Section \ref{s:due} is focused on the asymptotic behavior of the shape of the pile. 
We show that the solution $u(t,\cdot)$ converges, as $t\to +\infty$ to a limit $u_\infty$ (see Theorem \ref{t:conv}).
After the explicit computations of Examples~\ref{e:conv} and~\ref{e:conv2},
we discuss conditions for its convergence in finite time,
by formulating a conjecture and proving a result in this direction (see Conjecture~\ref{conj} and Theorem \ref{t:fintime}).

In Section \ref{s:tre} we show that $u_\infty$ is the $u$ component of a stationary solution of~\eqref{f:Pintro} 
(see Theorem~\ref{t:cm11}). This result is essentially based on a careful analysis of the stationary problem developed in \cite{CM11}.

\subsection*{Notations} \phantom{xx}

\noindent The Euclidean norm of $\xi \in \R^N$ will be denoted by $|\xi|$.

\noindent For $E\subseteq \R^N$, $\chi_E$ denotes the characteristic function of $E$, that is
\[
\chi_E(x)=
\begin{cases}
1, & \text{if}\ x\in E, \\
0 & \text{if}\ x\not\in E. 
\end{cases}
\]

\noindent For any $E\subset \R^N$, we denote by $\meas(E)$  the set of bounded Borel measures supported on $E$, and by $\meas^+(E)$ the set of non-negative measures in $\meas(E)$.

\noindent For $\mu\in L^\infty(0,T;\meas(E))$, we set $\mu_t=\mu(t,\cdot)$.

\noindent 
Given a function $u = u(t,x)$, $\partial_t u$ and $\nabla u$ denote respectively  the time derivative and the spatial part of the gradient.

\noindent For any open set $A$, $C^\infty_c(A)$ denotes the set of smooth functions with compact support in $A$, and $\mathcal{D}'(A)$ its topological dual, that is, the set of distributions on $A$.

\noindent $\lipu(A)$ is the set of Lipschitz functions in $\overline{A}$ with Lipschitz constant $1$, that is,
\[
\lipu(A)=\left\{ u\colon \overline{A} \to \R\colon u(x) - u(y) \leq |x-y|,\ \forall x,y\in\overline{A} \right\}.
\]

\noindent $L^1_+(A)$ is the set of non-negative functions in $L^1(A)$.

\noindent For $f \in L^1_+(A)$, $\spt(f)\subseteq\overline{A}$ denotes the
essential support of $f$ as a function extended in
$\R^N$ by setting $f=0$ on $\R^N\setminus A$, that is the complement in $\R^N$ of the largest open set where $f=0$ almost everywhere.

\section{The evolutionary problem} \label{s:uno}
In what follows we fix an integer $N\geq 1$, and
\begin{itemize}
	\item[(D1)] a non-empty open convex bounded set $\Omega\subset\R^N$;
	\item[(D2)] a lower semicontinuous function $\phi\colon\partial\Omega\to [0, +\infty[$;
	\item[(D3)] a non-negative integrable function $f \in L^1_+(\Omega)$.
\end{itemize}
We introduce the convex set of admissible profiles
\begin{equation}
	\label{f:Xp}
	\Xp :=
	\left\{
	u \in \lipu(\Omega)\colon
u\geq 0\ \text{in}\ \overline{\Omega},\
	u\leq \phi\ \text{on}\ \partial\Omega
	\right\},
\end{equation}
and we consider the evolutionary problem \eqref{f:Pintro} 
with the previous data and initial profile $u_0 = 0$.
More precisely,
we say that $(u,v,\nu)$ is a solution to the system
\begin{equation}
	\label{f:P}
	\begin{cases}
		\partial_t u - \dive(v\, \nabla u) = f - \nu
		&\text{in}\ \R^+\times\R^N,
		\\
		|\nabla u|\leq 1,\ v\geq 0, \ (1-|\nabla u|) v = 0,
		&\text{in}\ \R^+\times\Omega,
		\\
		0\leq u\leq \phi
		&\text{in}\ \R^+\times\partial\Omega,
		\\
		u = \phi
		& \text{$\nu$-a.e.\ in $\R^+\times\partial\Omega$},
		\\
		u(0, \cdot) = 0,
	\end{cases}
\end{equation}
if, for every $T>0$,
\begin{itemize}
	\item[(S1)] $u\in L^\infty(0,T; W^{1,\infty}(\Omega))$, $\partial_t u \in L^2( ]0,T[\times \Omega)$,
	$u(t,\cdot)\in\Xp$ for a.e.\ $t\in [0,T]$;
	\item[(S2)] $v\in L^\infty(0,T; L^1_+(\Omega))$, $\nu\in L^\infty(0,T; \meas^+(\partial \Omega))$;
	\item[(S3)] $(1-|\nabla u(t,x)|)\, v(t,x)=0$ for $\mathcal{L}^{N+1}$--a.e.\ $(t,x)\in ]0,T[ \times \Omega$;
	\item[(S4)] $u(0, \cdot) = 0$ in $\Omega$;
	\item[(S5)] $u(t,x)=\phi(x)$ $\nu_t$--a.e.\ on $\partial\Omega$, for  a.e.\ $t\in ]0,T[$;
	\item[(S6)] for every test function $\varphi\in C^\infty_c(\R^N)$, it holds that
	\[
	\frac{d}{dt}\int_{\Omega} u(t, x) \varphi(x) \, dx
	+ \int_{\Omega} v(t,x) \nabla u(t,x) \cdot \nabla\varphi(x)\, dx
	= \int_\Omega f(x)\varphi(x)\, dx
	- \int_{\partial\Omega} \varphi(x)\, d\nu_t(x),
	\quad
	\text{in}\ \mathcal{D}'(0,T).
	\]
\end{itemize}

\begin{remark}
\label{r:reg}
We remark that every function $u\in L^\infty(0,T; W^{1,\infty}(\Omega))$ belongs also to
$C([0,T]; L^2(\Omega))$ (see, \cite[Theorem 7.104]{Sal}), so that the initial condition $u(0,\cdot) = 0$ (or even $u(0,\cdot) = u_0\in\Xp$) makes sense.
Moreover, this implies also that the condition $u(t,\cdot) \in \Xp$ in (S1) holds for every $t\in [0,T]$.
\end{remark}

A crucial role in the description of the solutions to problem \eqref{f:P} is played by the Lax--Hopf function associated with the boundary datum $\phi$:
\begin{equation*}
	\uhl(x) := \min\{\phi(y) + |x-y|\colon y\in\partial\Omega\}, \qquad x\in\overline{\Omega}.
\end{equation*}
We recall that $\uhl$ is a Lipschitz function in $\overline{\Omega}$,
$|\nabla\uhl| = 1$ a.e.\ in $\Omega$,
and it is the maximal function in the set $\Xp$ defined in \eqref{f:Xp}, that is
\begin{equation}\label{f:maxX}
u \leq \uhl \ \text{in}\ \overline{\Omega}, \qquad \forall u\in\Xp.
\end{equation}

\begin{remark}\label{r:dist}
If $\phi=0$ (the open table problem in the variational models for growing sandpiles), then the Lax--Hopf function
is the distance function from the boundary of $\Omega$.
\end{remark}

For $x\in \Omega$ we introduce the set $\Pi(x)$ of all projection of $x$ on $\partial \Omega$, that is
\begin{equation*}
	\Pi(x):= \{y\in\partial \Omega\colon \uhl(x)=\phi(y) + |x-y|\},
\end{equation*}
and the discharge boundary
\begin{equation}
	\label{f:Gf}
	\Gamma_f := \{
	y\in\partial\Omega\colon \exists x\in\spt(f)
	\ \text{such that} \ y\in\Pi(x)
	\} = \bigcup_{x\in\spt(f)} \Pi(x).
\end{equation}
Since $\phi$ is a lower semicontinuous function, and $\spt(f)$ is compact, it is readily seen that $\Gamma_f$ is closed.

\smallskip

The following existence result for problem \eqref{f:P} and the properties of the solutions needed in the rest of the paper have been proved in \cite[Theorem~6.5]{DPJ} (see also \cite{Pr}).

\begin{theorem}
\label{t:exidp}
	Under the assumptions (D1)--(D3) there exists a solution $(u,v,\nu)$ of \eqref{f:P}. Moreover
	\begin{itemize}
		\item[(i)] the $u$ component of the solution is unique and $t \mapsto u(t,\cdot)$ is a non-decreasing function in $\R^+$;
		\item[(ii)] the measure $\nu_t$ is supported on $\Gamma_f$ for a.e. $t \geq 0$, and to every $\nu$ corresponds a unique $v$.
	\end{itemize}
\end{theorem}

\begin{remark}
	As a matter of fact, the result in \cite{DPJ} is obtained in a more general setting: the source $f$ is assumed to be a non-negative bounded measure in $\Omega$, and, in turn, the $v$ component is a non-negative bounded measure in $\Omega$, in general. Nevertheless, the $(v,\nu)$ components are obtained by a duality and optimal transport argument, and hence we can apply the regularity results for transport densities (see \cite[Theorem~4.13]{DPP2002}, or 
	\cite[Theorem~2]{San}), and recover the absolute continuity 
of the $v$ component with respect to the Lebesgue measure in $\Omega$.
\end{remark}

The following result, firstly proved in \cite{Pr}, and then detailed in \cite[Theorem 4.3]{DPJ} 
(see also \cite{BoBu}),  
shows that the PDEs system \eqref{f:P} may be considered as an equivalent first order condition for a constrained optimization problem solved by the $u$ component, in such a way the other components $(v,\nu)$ of the solution may be understood as Lagrange multipliers.

\begin{theorem}\label{t:incl}
The following equivalence holds.
\begin{itemize}
\item[(i)]
If $(u,v,\nu)$ is a solution of \eqref{f:P}, then for every $T>0$ it holds that
	\begin{equation}
		\label{f:var2}
\int_\Omega \left(f(x)-\partial_t u(t,x)\right)\left(w(x) - u(t,x) \right) \, dx\leq 0 \, \qquad \forall w\in\Xp,
	\end{equation}
for a.e.\  $t\in ]0,T[$.

\item[(ii)]
Let $u\in L^\infty(0,T;W^{1,\infty}(\Omega))$ satisfy (S1), the initial condition (S4),
and the maximality condition~\eqref{f:var2}.
Then there exists $(v,\nu)$ such that $(u,v,\nu)$ is a solution to \eqref{f:P}.
\end{itemize}
\end{theorem}

\begin{remark}
	Using the terminology of convex analysis, when $f\in L^2(\Omega)$, 
	the maximality condition~\ref{f:var2} 
	can be rephrased as a differential inclusion. 
	
	Specifically, let $I \colon L^2(\Omega) \to [0,+\infty]$  denote the indicator function of the convex set $\Xp$, defined by
	\[
	I(w) :=
	\begin{cases}
		0, & \text{if}\  w\in\Xp, \\
		+\infty, & \text{otherwise},
	\end{cases}
	\]
	and let us denote by $\partial I(w)$ its subdifferential at $w\in L^2(\Omega)$.

	Then, the variational inequality~\eqref{f:var2}  is equivalent to  
	the differential inclusion 
	\[
	f - \partial_t u(t,\cdot) \in \partial I(u(t,\cdot)), \qquad t\geq 0.
	\]
	In the following, we will say that $u$ satisfies $f - \partial_t u \in \partial I (u(t, \cdot))$
	if the maximization condition~\eqref{f:var2} holds.
\end{remark}

The variational inequality \eqref{f:var2} provides a great deal of information regarding the properties of the solution's $u$ component (as those in Theorem~\ref{t:exidp}(i)). To get them, it is useful to recall the following derivation rule, which 
has been proved in a more general setting in
\cite[Lemma~4.2]{DPJ}.

\begin{lemma}
\label{l:dpj}
Let $w\in L^1(0,T; W^{1,\infty}(\Omega))$ with $\partial_t w \in L^2(]0,T[\times\Omega)$.
Then
\[
\frac{1}{2}\frac{d}{dt} \int_\Omega |w(t,x)|^2 \, dx =
\int_\Omega w(t,x)\, \partial_t w(t,x)\, dx,
\qquad
\text{for a.e.}\ t\in ]0,T[.
\]
\end{lemma}

The uniqueness of the solution to \eqref{f:var2} with initial datum $u(0,\cdot) = u_0\in\Xp$, and its monotonicity with respect to $t$  are  consequences of
the following comparison principle, proved
in \cite[Lemma~3.1]{CCS} or \cite[Proposition~4.1]{DPJ},
and valid also for time-varying sources.
Since the result is relevant to our purposes, we include 
a sketch of its proof for completeness.

\begin{proposition}[Comparison Principle]
	\label{p:comp}
	Let $f_1, f_2 \in L^\infty(0,T; L^1_+(\Omega))$, $u_0^1, u_0^2\in\Xp$, with $f_1\geq f_2$, and
	$u_0^1 \geq u_0^2$.
	Let $u_i$, $i=1,2$, be the solutions to
	\begin{equation*}
		\begin{cases}
			f_i(t,\cdot) - \partial_t u_i(t,\cdot) \in \partial I(u_i(t,\cdot)),
			& \text{for a.e.}\ t\in ]0,T[,
			\\
			u_i(0,\cdot) = u_0^i.
		\end{cases}
	\end{equation*}
	Then $u_1 \geq u_2$ on $]0,T[\times\Omega$.
\end{proposition}

\begin{proof}
Let $u^+(t,x) := \max\{u_1(t,x), u_2(t,x)\}$ and $u^-(t,x) := \min\{u_1(t,x), u_2(t,x)\}$;
observe that $u^+, u^-\in\Xp$.
By the optimality of $u_1$ and $u_2$, using respectively $u^+$ and $u^-$
as competitors in~\eqref{f:var2}, we have that
\[
\begin{split}
\int_\Omega( f_1(t,x) - \partial_t u_1(t,x))( u^+(t,x) - u_1(t,x))\, dx & \leq 0, \\
\int_\Omega(  f_2(t,x) - \partial_t u_2(t,x))(u^- (t,x)- u_2(t,x))\, dx & \leq 0,
\end{split}
\qquad\text{a.e.\ } t\geq 0.
\]
Since $u^+ - u_1 = (u_2-u_1)\chi_{\{u_1 < u_2\}} = u_2 - u^-$
and $\chi_{\{u_1 < u_2\}}\, \partial_t u_2 = \chi_{\{u_1 < u_2\}}\, \partial_t u^+$,
we obtain that
\[
\int_\Omega(  f_2 (t,x)- \partial_t u^+(t,x))( u_1 (t,x)- u^+(t,x) )\, dx
= 
\int_\Omega(  f_2 (t,x)- \partial_t u_2(t,x)) ( u^- (t,x)- u_2(t,x) )\, dx
\leq 0,
\]
so that, by Lemma~\ref{l:dpj} and recalling that $f_2\geq f_1$,
\[
\begin{split}
\frac{1}{2} & \frac{d}{dt}   \int_\Omega |u^+(t,x) - u_1(t,x)|^2\, dx
 =
\int_\Omega( \partial_t u^+(t,x) - \partial_t u_1(t,x))( u^+(t,x) - u_1(t,x) )\, dx \\
& \leq
\int_\Omega f_2(t,x) (u^+(t,x) - u_1 (t,x))\, dx
-
\int_\Omega f_1(t,x) ( u^+(t,x) - u_1(t,x)) \, dx
\leq 0
\end{split}
\]
for a.e.\ $t\geq 0$.
By Remark \ref{r:reg},  $\psi(t) := \|u^+(t) - u_1(t)\|^2_{L^2(\Omega)}$ 
is a continuous function with $\psi(0) = 0$,
hence, from the above inequality we conclude that
$\psi\equiv 0$, i.e.\
$u^+(t, \cdot) = u_1(t,\cdot)$ for every $t\geq 0$.
\end{proof}

\begin{theorem}\label{t:uniq}
Given $f \in L^\infty(0,T; L^1_+(\Omega))$, and $u_0\in\Xp$, the solution to the problem
\begin{equation*}
	\begin{cases}
		f(t,\cdot) - \partial_t u(t,\cdot) \in \partial I(u(t,\cdot)),
		& \text{for a.e.}\ t\in ]0,T[,
		\\
		u(0,\cdot) = u_0,
	\end{cases}
\end{equation*}
is unique, and $t \mapsto u(t,\cdot)$ is a monotone non-decreasing function in $\R^+$.
\end{theorem}

\begin{proof}
	The uniqueness of the solution is a direct consequence of the comparison principle. 
	
	Moreover, fixed $t_0>0$, and applying the comparison principle to $f_1=f$, $f_2 \equiv 0$, and $u^1_0=u^2_0=u(t_0,x)$, we get $u(t,x) \geq u(t_0,x)$ for $t\geq t_0$.
\end{proof}

\section{Asymptotic stability of the profiles}\label{s:due}

Let $(u,v,\nu)$ be a solution to \eqref{f:P}. 
Since, by Theorem \ref{t:incl}(i), Theorem \ref{t:uniq}, and \eqref{f:maxX},  $t \mapsto u(t, \cdot)$ 
is a monotone non-decreasing function in $\R^+$ and
$0\leq u(t, \cdot) \leq \uhl$ for every $t$,
then there exists the limit
\begin{equation}\label{f:uinfty}
\ul(x) := \lim_{t\to +\infty} u(t,x),
\qquad x\in\overline{\Omega}.
\end{equation}
Moreover $\ul\in \Xp$, due to the fact that $u(t,\cdot)$ belongs to $\Xp$ for every $t \geq 0$, and the convergence in 
\eqref{f:uinfty} is uniform in $\overline{\Omega}$.

We aim to provide an explicit representation of $\ul$. 
As a first step, we show that the asymptotic profile is maximal where the source is active.

\begin{lemma}
\label{l:conv}
It holds that
$\ul(x) =\uhl(x)$
for every $x\in\spt(f)$.
\end{lemma}

\begin{proof}
Assume by contradiction that there exists
$x_0\in\spt(f)\subset\overline{\Omega}$ such that
$\ul(x_0) < \uhl(x_0)$.
Let $\delta := (\uhl(x_0) - \ul(x_0))/2$ and let $r>0$
be such that $\uhl(x) - \ul(x) \geq \delta$ for
every $x\in B_r(x_0) \cap\Omega$.
Since $f\in L^1_+(\Omega)$ and $x_0\in\spt(f)$,
we also have that
$\int_{B_r(x_0) \cap \Omega} f \, dx > 0$.
As a consequence,
for every $t\geq 0$ we have that
\begin{equation}
\label{f:dec}
\begin{split}
\int_\Omega f(x) & \, (\uhl(x) - u(t,x))\, dx  \geq
\int_\Omega f \, (\uhl (x)- \ul(x))\, dx \\ & \geq
\int_{B_r(x_0)\cap\Omega} f(x) \, (\uhl (x)- \ul(x))\, dx
\geq
\delta \int_{B_r(x_0)\cap\Omega} f (x)\, dx
=: \rho > 0,
\quad
\forall t\geq 0.
\end{split}
\end{equation}
By \eqref{f:var2}, using $\uhl\in\Xp$ as competitor, it
holds that
\[
\int_{\Omega} (f(x) - \partial_t u(t,x))(\uhl(x) - u(t,x))\, dx \leq 0.
\]
Hence, by Lemma~\ref{l:dpj},
the above inequality and~\eqref{f:dec}
yield
\[
\begin{split}
\frac{1}{2}  \frac{d}{dt} \int_\Omega |\uhl(x) - u(t,x)|^2\, dx
& = -\int_\Omega \partial_t u (t,x)\, (\uhl (x)- u(t,x))\, dx \\ &
\leq -\int_\Omega f(x) \, (\uhl (x)- u(t,x))\, dx
\leq -\rho,
\quad\forall t\geq 0,
\end{split}
\]
in  contradiction with the fact that $t\mapsto\|\uhl - u(t)\|^2_{L^2(\Omega)}$ is a non-negative continuous function.
\end{proof}

To proceed further with the description of $\ul$, which will turn out to be a stationary solution of the problem (see Section~\ref{s:tre}), we need some definitions.

\begin{definition}
A segment $\cray{y,x}$ is called a transport ray if
$y\in\partial\Omega$, $x\in\overline{\Omega}$,
$\uhl(x) = \uhl(y) + |x-y|$
(i.e.\ $y\in \Pi(x)$),
and $\cray{y,x}$  is not a proper subset of another segment
satisfying the same properties.
The points $y$ and $x$ are called, respectively,
the initial and the final point of the transport ray.

We denote by $J\subset\overline{\Omega}$ the set  of final points of the transport rays, defined by
\begin{equation}
	\label{f:J}
	J := \{
	x\in\overline{\Omega}\colon
	\exists y\in\partial\Omega
	\ \text{s.t.\ $\cray{y,x}$ is a transport ray}
	\}.
\end{equation}
\end{definition}

Given $f\in L^1_+(\Omega)$, we define the function 
\begin{equation*}
	u_f(x) := 0\vee \sup\{
	\uhl(x) - |x-z|\colon z\in\spt(f)
	\}.
\end{equation*}

Finally, recalling the definition of $\Gamma_f$ in \eqref{f:Gf}, we introduce the set of admissible profiles 
achieving the boundary datum $\phi$ on $\Gamma_f$
\begin{equation}
	\label{f:Xf}
	\Xf := \{u\in \Xp\colon u = \phi\ \text{on}\ \Gamma_f\}\,.
\end{equation}

The main features of the function $u_f$, proved in \cite[Proposition 5.3 and Theorem 5.5]{CM11}, are the following.

\begin{theorem}
\label{t:uf}
Assume $f\in L^1_+(\Omega)$, $f \not\equiv 0$. Then the following hold.
\begin{itemize}
	\item[(i)] $u_f\in\Xf$, and $u_f=\uhl$ on $\spt (f)$;
	\item[(ii)] every function $u\in\Xf$ such that $u=u_f$ on $\spt (f)$ satisfies $u_f\leq u\leq \uhl$ on $\overline{\Omega}$;
	\item[(iii)]  $u_f=\uhl$  in $\overline{\Omega}$ if and only if $J \subseteq \spt (f)$.
\end{itemize}
\end{theorem}

We are now ready to prove that the asymptotic profile $u_\infty$ is, in fact, $u_f$.

\begin{theorem}
\label{t:conv}
Under the assumptions (D1)--(D3), the (unique) $u$ component of the solution to \eqref{f:P}
converges monotonically and uniformly to the function $u_f$ as $t\to +\infty$.
\end{theorem}

\begin{proof}
By Lemma~\ref{l:conv} and Theorem~\ref{t:uf}(i) we have that
$\ul = \uhl = u_f$ of $\spt(f)$.
Hence, by Theorem~\ref{t:uf}(ii)
we also deduce that 
\begin{equation}\label{f:estul}
u_f \leq \ul \leq \uhl\ \text{in}\ \Omega.
\end{equation}

Observe that $u_f$ is a (stationary) solution of \eqref{f:var2};
specifically, since $u_f = \uhl$ on $\spt(f)$ we have that
\[
\int_\Omega f (w-u_f)\, dx = \int_{\spt(f)} f (w - \uhl)\, dx \leq 0,
\qquad \forall w \in \Xp.
\]
Hence,
by the Comparison Principle
with $f_1 = f_2 = f$, $u^1_0 = u_f$, $u^2_0 = 0$,
we conclude that
$u(t,x) \leq u_f(x)$, $x\in\Omega$, for every $t\geq 0$,
so that $\ul \leq u_f$, which, together with \eqref{f:estul}, concludes the proof.
\end{proof}

\begin{remark}
\label{r:conv}
With a minor modification in the proof,
we can prove that the unique solution $u(t,\cdot)$ to~\eqref{f:var2}
with initial data $u_0\in\Xp$
converges monotonically and uniformly to
$\ul=u_0 \vee  u_f$ as $t\to +\infty$. 
\end{remark}

Since we consider here a source $f$ constant in time, 
one might be led to think that the evolution $u(t,\cdot)$ converges to $u_f$ in finite time.
Nevertheless, in the following examples we show that this
is not in general true.

\begin{example}
\label{e:conv}
Let $\Omega = B_1$ be the unit ball in $\R^N$ centered at the origin,
let $\phi \equiv 0$ and
$f(x) = (N+\alpha) |x|^\alpha$, $\alpha > 0$.
Since $\spt(f) = \overline{\Omega}$, by Theorem \ref{t:uf}(iii) and Theorem \ref{t:conv} (see also Remark \ref{r:dist}),
the limit function $\ul$ coincides with
the distance function from the boundary of $B_1$, i.e.\
\[
\ul(x) = \uhl(x) = 1 - |x|, \qquad x\in B_1.
\]
Starting from $u_0 =0$, in finite time 
$t_\alpha = \dfrac{2^\alpha - 1}{(N+1)\alpha}$
we reach the profile
$u_1(x) = \overline{u}_1(|x|)$,
with
\begin{equation}
\label{f:u1}
\overline{u}_1 (r) := 
\frac{1}{2} - \left|r - \frac{1}{2}\right|\,,
\qquad
r\in [0,1].
\end{equation}
Since the evolution for $t\in [0,t_\alpha]$ is not essential for the sake of our example,
we omit the related computations, and we assume to start at time $t=0$
from this initial profile $u_1$.
Let $\rho \colon [0,+\infty[ \to [0, +\infty[$ be the
unique non-negative solution to the
Cauchy problem
\begin{equation}
\label{f:oderho}
\begin{cases}
\dot{\rho} = - \frac{N}{2}\, \rho^\alpha,
\\
\rho(0) = \frac{1}{2}\,.
\end{cases}
\end{equation}
For every $\alpha > 0$, the solution is a monotone non-increasing function
and converges to $0$ as $t\to +\infty$.
If $\alpha \geq 1$ the solution is strictly positive
and strictly decreasing, while for every $\alpha \in ]0,1[$
there exists $\tau_\alpha > 0$ such that
$\rho(t) > 0$ for $t \in [0, \tau_\alpha[$ and
$\rho(t) = 0$ for every $t \geq \tau_\alpha$.

We claim that the function
\begin{equation}
\label{f:u}
u(t,x) :=
\begin{cases}
1 - 2\rho(t) + |x|,
& \text{if}\ |x| < \rho(t),
\\
1- |x|,
& \text{if}\ \rho(t) \leq |x| \leq 1,
\end{cases}
\end{equation}
solves~\eqref{f:var2} with $u(0,\cdot)=u_1$, and hence, by Theorem \ref{t:incl}(ii), $u(t,x)$ is the evolution of the profile of standing layer for $\phi \equiv 0$ and
$f(x) = (N+\alpha) |x|^\alpha$.
Clearly, $u(t,\cdot)$ converges to $\uhl$ in a finite time $\tau$ if and only if $\rho(t) = 0$
for every $t\geq \tau$.
Hence, if $\alpha\geq 1$ then we have that $u(t,0) < 1=\uhl(0)$ for every $t\geq 0$,
so that we do not have convergence to $\uhl$ in finite time.
If $\alpha\in ]0,1[$, instead, the source has enough mass on every small ball centered at the origin to
force the solution to converge to $\uhl$ in finite time $\tau_\alpha$. 

\smallskip

It remains to prove the claim:
for every $t\geq 0$ and every $w\in \Xp$,
it holds that
\[
K (t):= \int_{B_1} (f(x) -\partial_t u(t,x)) \cdot
(w(x) - u(t,x))\, dx \leq 0, \qquad \forall t \geq 0.
\]
Computing
\[
\partial_t u(t,x) =
\begin{cases}
- 2\dot{\rho}(t) = N\, \rho(t)^\alpha,
& \text{if}\ |x| < \rho(t),
\\
0,
& \text{if}\ \rho(t) < |x| < 1,
\end{cases}
\]
we deduce that
\[
K(t) = \int_{B_{\rho(t)}} ((N+\alpha) |x|^{\alpha} - N\, \rho(t)^\alpha)
\cdot (w(x) - 1 + 2 \rho(t) - |x|)\, dx
+ \int_{B_1 \setminus B_{\rho(t)}} f(x) (w(x) - \uhl(x))\,dx.
\]
Since $w \leq \uhl$, the  integral in $B_1 \setminus B_{\rho(t)}$ is non-positive.
In order to estimate the first integral,
we can use the following inequality,
which takes into account that $|\nabla w| \leq 1$:
setting $\hat{x} := x / |x|$ for every $x\neq 0$,
for every $\rho_0 \in [0,1]$
it holds that
\begin{equation}
\label{f:disw}
w(x)
\geq \widetilde{w}(x)  := w(\rho_0 \hat{x}) + |x| - \rho_0,
\quad\text{if}\ 0 < |x| \leq \rho_0,
\qquad
w(x)
\leq \widetilde{w}(x),
\quad
\text{if}\ \rho_0 \leq |x| \leq 1.
\end{equation}

Let $\rho_0(t) \in [0, \rho(t)]$ be defined by
\[
\rho_0(t) := \left(\frac{N}{N+\alpha}\right)^{1/\alpha} \rho(t),
\]
so that 
$ (N+\alpha) |x|^{\alpha}- N\rho(t)^\alpha$ is negative
if $|x| < \rho_0(t)$ and positive if $|x| > \rho_0(t)$.
Hence, from~\eqref{f:disw} we deduce that
\[
[(N+\alpha) |x|^{\alpha} - N\, \rho(t)^\alpha] w(x)
\leq
[(N+\alpha) |x|^{\alpha} - N\, \rho(t)^\alpha] \widetilde{w}(x),
\quad
\forall |x| < \rho(t),
\]
so that
\[
\begin{split}
K(t) & \leq \int_{B_{\rho(t)}} [(N+\alpha) |x|^{\alpha} - N\, \rho(t)^\alpha]
\cdot [\widetilde{w}(x) - 1 + 2 \rho(t) - |x|]\, dx
\\
& =
\int_{B_{\rho(t)}} [(N+\alpha) |x|^{\alpha} - N\, \rho(t)^\alpha]
\cdot [w(\rho_0(t)\hat{x}) - \rho_0(t) - 1 + 2 \rho(t)]\, dx
\\
& =
N\omega_N  [\overline{w}(\rho_0(t)) - \rho_0(t) - 1 + 2 \rho(t)] \int_0^{\rho(t)} 
[(N+\alpha) r^{\alpha} - N\, \rho(t)^\alpha]
\, r^{N-1}\, dr
= 0,
\end{split}
\]
where $\overline{w}(\rho_0) := \frac{1}{N \omega_N}\int_{S^{N-1}} w(\rho_0 \, \sigma)\, d\sigma$,
concluding the proof of our claim.
\end{example}
 
The construction of the evolution provided in Example~\ref{e:conv} above may seem somewhat mysterious.
Specifically, the role played by the ordinary differential equation~\eqref{f:oderho} should be clarified.
The following construction considering a more general source term could be useful to this aim.

\begin{example}
\label{e:conv2}
Consider a situation similar to the one described in Example~\ref{e:conv}, i.e.\
$\Omega = B_1\subseteq \R^N$ and $\phi \equiv 0$,
but with
\[
f(x) = \fr(|x|),
\quad
\fr\colon [0,1] \to \R
\ \text{continuous, increasing, with}\
\fr(0) = 0
\]
and assume that, at time $t=0$, we start with the profile
$u_1$ defined in~\eqref{f:u1}.

We want to prove that the function $u(t,x)$, defined in~\eqref{f:u},
is a solution to~\eqref{f:var2} for a suitable choice of
the function $\rho\colon [0,+\infty[ \to [0,+\infty[$.

In order to determine $\rho(t)$ (specifically, in order to get an ordinary differential equation
replacing \eqref{f:oderho}) we make the following consideration, based on the phenomenology of growing sandpiles:
only the mass poured in $B_{\rho(t)}$ is going to be incorporated in the pile, while the mass poured in $B_1 \setminus B_{\rho(t)}$
freely rolls downhill, since the profile is maximal. This translates into the additional condition
\[
 \int_{B_{\rho(t)}} \partial_t u(t,x)\, dx = \int_{B_{\rho(t)}} f(x)\, dx\,,\qquad t \geq 0
\]
Since $\partial_t u(t,x) = -2\dot{\rho}(t)$, this condition yields
\begin{equation}
\label{f:utmean}
-2 \dot{\rho}(t) = \frac{1}{|B_{\rho(t)}|} \int_{B_{\rho(t)}} f(x)\, dx, \qquad t \geq 0,
\end{equation}
with the convention that the mean value of the source $f$ on the ball $B_{\rho(t)}$,
appearing at the right-hand side, is equal to $0$ if $\rho(t) = 0$.
Taking into account the initial condition $\rho(0) = 1/2$,
the corresponding Cauchy problem has a unique non-negative solution $\rho(t)$.
We have two possibilities: either $\rho(t) > 0$ for every $t\geq 0$,
or there exists $\tau > 0$ such that $\rho(t) = 0$ for every $t\geq \tau$.
The second case, which corresponds to convergence in finite time,
happens if and only if
\begin{equation}
\label{f:tfin}
\tau:=\int_0^{1/2} \frac{|B_{\rho}|}{\int_{B_{\rho}} f(x)\, dx}\, d\rho
< +\infty.
\end{equation}
Now we choose $\rho(t)$ satisfying \eqref{f:utmean} and such that $\rho(0)=1/2$, and we show, as in Example~\ref{e:conv}, that the function $u$ defined in \eqref{f:u} is the solution to the variational inequality \eqref{f:var2} such that $u(0,\cdot)=u_1$, so that $u$ is the profile of the growing sand pile.

First of all,
since $f$ is a continuous function,
for every $t \geq 0$ there exists $\rho_0(t) \in [0,\rho(t)]$ such that
\begin{equation}
\label{f:fr}
\fr(\rho_0(t) ) = \frac{1}{|B_{\rho(t)}|} \int_{B_{\rho(t)}} f(x)\, dx.
\end{equation}
As a consequence, recalling~\eqref{f:utmean}, it holds that
\[
\partial_t u(t,x) = -2 \dot{\rho}(t) = \fr(\rho_0(t) ),
\qquad \forall t\geq 0.
\]
At this point we can complete the proof as in Example~\ref{e:conv}:
specifically, 
since $\fr$ is increasing we have that
\[
\fr(r) \leq \fr(\rho_0(t)),
\quad \forall r \in [0,  \rho_0(t)],
\qquad
\fr(r) \geq \fr(\rho_0(t)),
\quad \forall r\in [\rho_0(t), \rho(t)],
\]
so that, for every $w\in \Xp[0]$, it holds that
\[
\begin{split}
\int_{B_1} (f(x) -\partial_t u(t,x)) &
(w(x) - u(t,x))\, dx
 \leq 
\int_{B_{\rho(t)}} (\fr(|x|) -\fr(\rho_0(t) ))
(\widetilde{w}(x) - 1 + 2 \rho(t) - |x|)\, dx
\\
& =
N\omega_N  [\overline{w}(\rho_0(t)) - \rho_0(t) - 1 + 2 \rho(t)] \int_0^{\rho(t)} 
(\fr(r) -\fr(\rho_0(t) ) )
\, r^{N-1}\, dr
= 0,
\end{split}
\]
where $\widetilde{w}$ is defined in~\eqref{f:disw}, and the last equality follows from~\eqref{f:fr}.

Summarizing, the evolution $u(t,\cdot)$ in \eqref{f:u} with $\rho$ satisfying \eqref{f:utmean} and $\rho(0)=1/2$ converges
to $\ul=\uhl$ in finite time if and only if \eqref{f:tfin} holds.
\end{example}

The main point in Examples~\ref{e:conv} and~\ref{e:conv2} is that, in a unit time interval, only
a vanishing fraction (as $t\to +\infty$) of the mass poured by the source $f$
is added to the pile, while the remaining portion is discharged at the boundary
of the table.
In Example~\ref{e:conv},
when $\alpha \geq 1$, this infinitesimal fraction cannot fill the maximal profile
$\uhl$ in finite time. 
This happens because the source is too weak near the set of final points of the transport rays
$J = \{0\}$, meaning that condition~\eqref{f:tfin} is not satisfied.
Specifically,
\[
\int_0^{1/2} \frac{|B_{\rho}|}{\int_{B_{\rho}} f(x)\, dx}\, d\rho
= \int_0^{1/2} \rho^{-\alpha}\, d\rho = +\infty,
\qquad \forall \alpha \geq 1.
\]
These considerations lead us to state the following
\begin{conjecture}
\label{conj}
Assume that $f\in L^1_+(\Omega)$
satisfies
\begin{equation}
\label{f:dens}
\sup_{y\in J}\int_0^{1} \frac{|B_{\rho}|}{\int_{B_{\rho}(y)\cap\Omega} f(x)\, dx}\, d\rho
< +\infty.
\end{equation} 
Then $u(t, \cdot)$ converges to $\uhl$ in finite time.
\end{conjecture}

A condition stronger than \eqref{f:dens}, but perhaps easier to handle, is
the following:
There exist $\alpha \in [0,1[$ and a constant $c>0$ such that
\begin{equation}
\label{f:dens2}
\liminf_{r \searrow 0} \frac{1}{r^{N+\alpha}} \int_{B_r(y)\cap\Omega} f(x)\, dx \geq c,
\qquad
\forall y\in J.
\end{equation}
Clearly both conditions can be stated also in the case $f\in\meas^+(\Omega)$.
For example, \eqref{f:dens2} requires that the lower $(N+\alpha)$-dimensional density
of the measure $f$ restricted on $J$ is bounded from below by a positive constant.

We remark that, if $y\in J$ does not belong to the support of $f$,
then the integrand in \eqref{f:dens} is $+\infty$ for $\rho$ small enough,
hence, \eqref{f:dens} implies in particular that $J\subset\spt(f)$.
By Theorems~\ref{t:conv} and~\ref{t:uf}(iii), we already know that this
inclusion is necessary and sufficient in order to have that $u(t,\cdot)$
converges to $\uhl$ as $t\to +\infty$.

In the following theorem we prove a sufficient condition
in order to have convergence in finite time
to $\uhl$ (see also \cite[Theorem~3.3]{CCS} for a similar condition
in the case $\phi \equiv 0$).
We remark that, under the assumptions of Theorem~\ref{t:fintime} below,
condition~\eqref{f:dens2} is satisfied with $\alpha = 0$ and $c = \varepsilon$.

\begin{theorem}[Convergence in finite time]\label{t:fintime}
Let $J\subset\overline{\Omega}$ be the set of
final points of the transport rays
defined in \eqref{f:J}.
Assume that there exist
$r > 0$ and $\varepsilon \in ]0,r]$ such that
\[
f(x) \geq \varepsilon
\qquad
\forall x\in \Omega \cap \bigcup_{y\in J} B_r(y).
\]
Then $u(t, \cdot)$ converges to $\uhl$ in finite time.
\end{theorem}

\begin{proof}
Let $y \in J\cap\Omega$, and assume that $B_r(y) \subset\Omega$.
Consider the time-dependent source
\[
f_2(t,x) =
\begin{cases}
\varepsilon(0\vee (r - |x-y|)),
&\text{if}\ t \in [0, r/\varepsilon],
\\
\varepsilon \, \chi_{B_r(y)},
&\text{if}\ t > r/\varepsilon,
\end{cases}
\]
and let $u_2$ be the solution to~\eqref{f:var2},
with $f = f_2$ and $u_2(0,\cdot) = 0$.
Let us define
\[
\alpha(t) := u_2(t,y),
\quad
\uu(t,x) := 0 \vee (\alpha(t) - |x-y|),
\qquad x\in\Omega,\ 
t\geq 0.
\]
By Theorem \ref{t:uniq}, $\alpha$ is a continuous and non-decreasing function.
Let $\overline{t} \geq r/\varepsilon$ be the first time
such that $\overline{B}_{\alpha(t)}(y) \cap \partial\Omega\neq \emptyset$.
For $t\in [0,\overline{t}]$ we can check that 
$u_2(t,x) = \uu(t,x)$
and
\begin{equation}
\label{f:alpha}
\alpha(t) =
\begin{cases}
\varepsilon\, t,
& \text{if}\ t\in [0, r/\varepsilon],
\\
\overline{\alpha}(t)\,,
& \text{if}\ t\in [r/\varepsilon, \overline{t}],
\end{cases}
\end{equation}
with
\begin{equation}
\label{f:abar}
\overline{\alpha}(t) := \left[
r^{N+1} + (N+1)\varepsilon r^N \left(t - \frac{r}{\varepsilon}\right)
\right]^\frac{1}{N+1}\,,
\qquad t\geq 0
\end{equation}
(see the proof of Theorem~3.3 in \cite{CCS} for details).
Specifically, the representation \eqref{f:abar}, once we take for granted that
$u_2 = \uu$ for $t\in [0,\overline t]$, can be obtained as follows. 
The height $\alpha(t)$ of the cone can be computed 
taking into account that no sand can fall down the table before time $\overline{t}$,
so that the following balance of mass must hold:
\begin{equation}
\label{f:mass}
\int_\Omega f_2(t,x)\, dx = \int_\Omega \partial_t u_2(t,x)\, dx.
\end{equation}
A straightforward computation gives
\[
\int_\Omega f_2(t,x)\, dx =
\begin{cases}
\dfrac{\omega_N \varepsilon\, r^N}{N+1}\,,
& \text{if}\ t\in [0, r/\varepsilon],
\\[6pt]
\omega_N \varepsilon\, r^N
& \text{if}\ t\in [r/\varepsilon, \overline{t}],
\end{cases}
\qquad
\int_\Omega \partial_t u_2(t,x)\, dx
= 
\begin{cases}
\dfrac{\omega_N r^N\, \dot{\alpha}(t)}{N+1}\,,
& \text{if}\ t\in [0, r/\varepsilon],
\\[6pt]
\omega_N \alpha(t)^N\, \dot{\alpha}(t)\,,
& \text{if}\ t\in [r/\varepsilon, \overline{t}],
\end{cases}
\]
so that \eqref{f:alpha} follows observing that the function $\overline{\alpha}$,
defined in~\eqref{f:abar}, is the solution to
the Cauchy problem
\begin{equation}
\label{f:cpalpha}
\begin{cases}
\dfrac{d}{dt} \overline{\alpha}(t) = \dfrac{\varepsilon\, r^N}{\alpha(t)^N},
& t\geq 0,
\\
\overline{\alpha}(r/\varepsilon) = r\,.
\end{cases}
\end{equation}

For $T>\overline{t}$ large enough,
let $\nu\in L^\infty(0,T;\meas^+(\partial\Omega))$ be a measure
associated with $u_2$ through Theorem~\ref{t:incl}(ii), and let
\[
T(y) := \sup\{t\geq\overline{t}\colon \nu_t = 0\}\,.
\]
By \cite[Theorem~5.4]{DPJ}
we have that $u_2(t,x) = \uu(t,x)$ for $t\in [0, T(y)]$.
(To be precise, Theorem~5.4 in \cite{DPJ} has been proved in the case of a source $f_2$ 
which is a Dirac delta in $y$ but, for $t \leq T(y)$, that solution coincides with $u_2$.)  
In particular, there exists $z\in\partial\Omega$ such that
$\alpha(T(y)) - |z-y| = \phi(z)$.
Let us define 
\[
A(t) := \{x\in\Omega\colon \alpha(t) - |x-y| > 0\}
\subseteq B_{\alpha(t)}(y),
\qquad t\geq 0.
\]
Since 
the balance of mass~\eqref{f:mass} holds also for $t\in [\overline{t}, T(y)]$,
we deduce that
\[
\dot{\alpha}(t) = \dfrac{\omega_N \varepsilon\, r^N}{|A(t)|}\,,
\qquad t\in [r/\varepsilon, T(y)].
\]
Observing that $|A(t)| \leq |B_{\alpha(t)}| = \omega_N \alpha(t)^N$,
by comparison with the solution $\overline{\alpha}$ of~\eqref{f:cpalpha}
we conclude that
\[
\alpha(t) \geq \overline{\alpha}(t),
\qquad \forall t\in [r/\varepsilon, T(y)].
\]
We claim that $u_2(T(y), y) = \uhl(y)$. Namely,
\[
u_2(T(y), y)=\alpha(T(y)) 
= \phi(z) + |z-y| \geq \uhl(y),
\]
so that the claim follows from the maximality of $\uhl$.

From the relation
\[
\uhl(y) = \alpha(T(y)) \geq \overline{\alpha}(T(y))
\]
we conclude that
\[
T(y) \leq \frac{\uhl(y)^{N+1} + N\, r^{N+1}}{(N+1)\varepsilon\, r^N} \leq 
\tau :=
\dfrac{[\min_{\partial\Omega}\phi + \diam(\Omega)]^{N+1} + N\, r^{N+1}}{(N+1) \varepsilon\, r^N}\,.
\]
Finally, since $f\geq \varepsilon$ in a tubular  neighborhood of $J$, then  $f\geq f_2$ and hence,
by the comparison principle and 
the fact that $u(T(y), y) = \uhl(y)$,
we get that $u(t, y) = \uhl(y)$ for every $t\geq \tau$.

If $y\in J$ but $B_r(y)$ is not contained in $\Omega$,
we can modify the above proof taking into account that
$\Omega$ satisfy a uniform interior cone condition.
This implies that there exists a positive constant $\gamma$
such that $|B_r(y) \cap \Omega| \geq \gamma |B_r(y)|$
for every $y\in \overline{\Omega}$,
and it can be proved that 
$\alpha(t) := u_2(t,y) \geq \gamma\, \overline{\alpha}(t)$
for $t \in [r/\varepsilon, T(y)]$,
so that we can obtain a uniform estimate from above on the time
$T(y)$ defined above.

Hence there exists a time $\tau' \geq \tau$ such that
$u(t, y) = \uhl(y)$ for every $t\geq \tau'$ and for
every $y\in J$,
so that the conclusion follows from \cite[Theorem~5.5]{CM11}.
\end{proof}

\section{Stationary solutions}\label{s:tre}

In this section we study the stationary solutions
of~\eqref{f:P}, i.e., the solutions of
\begin{equation}
\label{f:Ps}
\begin{cases}
- \dive(v\, \nabla u) = f - \nu
&\text{in}\ \R^N,
\\
|\nabla u| \leq 1,\ v\geq 0, \ (1-|\nabla u|) v = 0,
&\text{a.e.\ in}\ \Omega,
\\
0\leq u \leq \phi
&\text{in}\ \partial\Omega,
\\
u(x) = \phi(x)
& \text{$\nu$-a.e.\  $x\in\partial\Omega$},
\end{cases}
\end{equation}
where $f\in L^1_+(\Omega)$ and
$\phi\colon\partial\Omega\to [0,+\infty[$ are as
in the previous sections.
More precisely, $(u,v,\nu)$ is called a stationary solution to~\eqref{f:Ps}
if $u\in \Xp$, $v\in L^1_+(\Omega)$, $\nu\in\meas^+(\partial\Omega)$
satisfy
$(1-|\nabla u|) v = 0$ a.e.\ in $\Omega$, $u=\phi$ $\nu$-a.e.\ on $\partial \Omega$, and
\begin{equation}
\label{f:weaknu}
\int_\Omega
v\, \nabla u \cdot \nabla\psi\, dx
= \int_\Omega f\, \psi\, dx
- \int_{\partial\Omega} \psi\, d\nu,
	\qquad
	\forall\psi\in C^\infty_c(\R^N).
\end{equation}

We first recall the stationary version of Theorem~\ref{t:incl},
which has been proved in \cite[Theorem~3.2]{DPJ}.

\begin{theorem}\label{t:incls}
The following equivalence holds.
\begin{itemize}
\item[(i)]
If $(u,v,\nu)\in \Xp\times L^1_+(\Omega)\times \meas^+(\partial\Omega)$ 
is a solution of \eqref{f:Ps}, then 
	\begin{equation}
		\label{f:var3}
\int_\Omega f(x) \left(
w(x)-u(x) \right) \, dx\leq 0 \, \qquad \forall w\in\Xp.
	\end{equation}

\item[(ii)]
If $u\in \Xp$ satisfies
the maximality condition~\eqref{f:var3},
then there exists $(v,\nu)\in L^1_+(\Omega)\times\meas^+(\partial\Omega)$ 
such that $(u,v,\nu)$ is a solution to~\eqref{f:Ps}.
\end{itemize}
\end{theorem}

We now prove that
problem~\eqref{f:Ps} can be reformulated without making any reference
to the measure $\nu\in\mathcal{M}^+(\partial\Omega)$.

\begin{theorem}[Equivalent formulation for stationary solutions]
\label{t:equiv}
Let $\Omega, \phi, f$ satisfy (D1), (D2), (D3), let
$\Gamma_f\subset\partial\Omega$ be the set defined in~\eqref{f:Gf}, and 
let $\Xf\subset W^{1,\infty}(\Omega)$
be the set defined in \eqref{f:Xf}.
Then the following equivalence holds.
\begin{itemize}
\item[(i)]
If $(u,v,\nu)\in \Xp\times L^1_+(\Omega)\times \meas^+(\partial\Omega)$
is a solution to~\eqref{f:Ps},
then $(u,v)$ is a solution to 
\begin{equation}
\label{f:Ps2}
\begin{cases}
- \dive(v\, \nabla u) = f
&\text{in}\ \Omega,
\\
|\nabla u| \leq 1,\ v\geq 0, \ (1-|\nabla u|) v = 0,
&\text{a.e.\ in}\ \Omega,
\\
0\leq u \leq \phi
&\text{in}\ \partial\Omega,
\\
u(x) = \phi(x)
& \text{on}\ \Gamma_f,
\end{cases}
\end{equation}
meaning that $(u,v)\in \Xp[f]\times L^1_+(\Omega)$ satisfies
$(1-|\nabla u|) v = 0$ a.e.\ in $\Omega$ and
\begin{equation}
\label{f:weak}
\int_\Omega v\, \nabla u \cdot \nabla\psi\, dx
= \int_\Omega f\psi\, dx,
\qquad
\forall \psi \in C^\infty_c (\R^N \setminus \Gamma_f).
\end{equation}

\item[(ii)]
If $(u,v)\in \Xp[f]\times L^1_+(\Omega)$ is a solution to~\eqref{f:Ps2},
then there exists a measure $\nu\in\meas^+(\partial\Omega)$ such
that $(u,v,\nu)$ is a solution to~\eqref{f:Ps}.
\end{itemize}
\end{theorem}

Before proving Theorem~\ref{t:equiv},
we recall that
the problem of existence and uniqueness of solutions
to~\eqref{f:Ps2} has been analyzed in detail in
\cite{CM11}, in the more general case of a convex constraint
on $\nabla u$ and for a non-convex domain $\Omega$.

The main results in \cite{CM11} relevant to our problem can be summarized
in the following theorem.

\begin{theorem}\label{t:cm11}
Under the assumptions of Theorem \ref{t:equiv}, the following holds.
\begin{itemize}
\item[(i)]
[Existence] There exists a unique $v_f\in L^1_+(\Omega)$ such that
$(\uhl, v_f)$ is a solution to \eqref{f:Ps2}.

\item[(ii)]
[Uniqueness of $v$ and characterization of $u$]
A pair $(u,v)\in \Xf\times L^1_+(\Omega)$ is a solution
to \eqref{f:Ps2} if and only if $v = v_f$ and
$u_f\leq u \leq \uhl$.

\item[(iii)]
[Uniqueness]
$(\uhl, v_f)$ is the unique solution to \eqref{f:Ps2}
if and only if $J\subseteq \spt(f)$,
where $J\subset\overline{\Omega}$ is the set of
final points of the transport rays
defined in \eqref{f:J}.
\end{itemize}
\end{theorem}

\begin{proof}[Proof of Theorem~\ref{t:equiv}]
(i)  By Theorem~\ref{t:exidp} we already know that $\spt(\nu) \subseteq\Gamma_f$, hence
$\int_{\partial\Omega} \psi\, d\nu = 0$ for every $ \psi \in C^\infty_c (\R^N \setminus \Gamma_f)$,
so that~\eqref{f:weak} easily follows from~\eqref{f:weaknu}.
It remains to prove that $u=\phi$ on $\Gamma_f$.
Since $(u,v,\nu)$ is a solution to~\eqref{f:Ps}, $u$ satisfies
$\int_\Omega f(w-u)\,dx \leq 0$ for every $w\in\Xp$.
Choosing $w = \uhl$ we deduce that $u = \uhl$ on $\spt(f)$. 
Let $y\in\Gamma_f$. By definition~\eqref{f:Gf} of $\Gamma_f$, there exists $x\in\spt(f)$ such that
$\uhl(x) = u(x) = \phi(y) + |x-y|$,
hence
$u(y) \geq u(x) - |x-y| \geq \phi(y)$.
Since $u\leq \phi$ on $\partial\Omega$, we conclude that $u(y) = \phi(y)$.

(ii) Let $(u,v)\in\Xp[f]\times L^1_+(\Omega)$ satisfy~\eqref{f:weak} and $(1-|\nabla u|) v = 0$ a.e.\ in $\Omega$.
By Theorem~\ref{t:cm11}(ii) we have that $v = v_f$ and $u_f\leq u \leq \uhl$.
By Theorem~\ref{t:uf}(i) we also know that $u=\uhl$ on $\spt(f)$, hence
\[
\int_\Omega f (w - u)\, dx
= \int_{\spt(f)} f (w - \uhl)\, dx \leq 0,
\qquad
\forall w\in\Xp, 
\]
so that \eqref{f:var3} holds, and the conclusion follows
from Theorem~\ref{t:incls}(ii). 
\end{proof}

Comparing the information about evolution provided by Theorem \ref{t:exidp} 
with those about stationary solutions provided by Theorem \ref{t:cm11}, 
we notice that during the evolution, the dynamic of the standing layer is unique, 
while the rolling layer can assume different configurations, 
while for stationary configurations, the standing layer can vary 
(retaining memory of the initial configuration of the pile), while the rolling layer remains fixed.

\def\cprime{$'$}
\providecommand{\bysame}{\leavevmode\hbox to3em{\hrulefill}\thinspace}
\providecommand{\MR}{\relax\ifhmode\unskip\space\fi MR }
\providecommand{\MRhref}[2]{%
  \href{http://www.ams.org/mathscinet-getitem?mr=#1}{#2}
}
\providecommand{\href}[2]{#2}


\begin{thebibliography}{99}

\bibitem{BoBu}
{G.} Bouchitt\'e and {G.} Buttazzo, \emph{Characterization of optimal shapes
  and masses through {M}onge-{K}antorovich equation}, J.\ Eur.\ Math.\ Soc.
  \textbf{3} (2001), 139--168.

\bibitem{CCCG}
P.~Cannarsa, P.~Cardaliaguet, G.~Crasta, and E.~Giorgieri, \emph{A boundary
  value problem for a {PDE} model in mass transfer theory: representation of
  solutions and applications}, Calc. Var. Partial Differential Equations
  \textbf{24} (2005), no.~4, 431--457. \MR{2180861}

\bibitem{CCS}
{P.} Cannarsa, {P.} Cardaliaguet, and {C.} Sinestrari, \emph{On a differential
  model for growing sandpiles with non-regular sources}, Comm. Partial
  Differential Equations \textbf{34} (2009), no.~7-9, 656--675. \MR{MR2560296}

\bibitem{CFV}
{G.} Crasta and {S.} Finzi~Vita, \emph{An existence result for the sandpile
  problem on flat tables with walls}, Netw. Heterog. Media \textbf{3} (2008),
  no.~4, 815--830. \MR{2448942}

\bibitem{CM6}
{G.} Crasta and {A.} Malusa, \emph{The distance function from the boundary in a
  {M}inkowski space}, Trans. Amer. Math. Soc. \textbf{359} (2007), no.~12,
  5725--5759. \MR{2336304}

\bibitem{CM7}
{G.} Crasta and {A.} Malusa, \emph{On a system of partial differential
  equations of {M}onge-{K}antorovich type}, J. Differential Equations
  \textbf{235} (2007), no.~2, 484--509. \MR{2317492}

\bibitem{CM9}
{G.} Crasta and {A.} Malusa, \emph{A sharp uniqueness result for a class of
  variational problems solved by a distance function}, J. Differential
  Equations \textbf{243} (2007), no.~2, 427--447. \MR{2371795}

\bibitem{CM8}
{G.} Crasta and {A.} Malusa, \emph{A variational approach to the macroscopic
  electrodynamics of antisotropic hard superconductors}, Arch. Ration. Mech.
  Anal. \textbf{192} (2009), no.~1, 87--115. \MR{2481062}

\bibitem{CM10}
{G.} Crasta and {A.} Malusa, \emph{A nonhomogeneous boundary value problem in
  mass transfer theory}, Calc. Var. Partial Differential Equations \textbf{44}
  (2012), no.~1-2, 61--80. \MR{2898771}

\bibitem{CM11}
{G.} Crasta and {A.} Malusa, \emph{Existence and uniqueness of solutions for a
  boundary value problem arising from granular matter theory}, J. Differential
  Equations \textbf{259} (2015), no.~8, 3656--3682. \MR{3369258}

\bibitem{DPJ}
{L.} De~Pascale and {C.} Jimenez, \emph{Duality theory and optimal transport
  for sand piles growing in a silos}, Adv. Differential Equations \textbf{20}
  (2015), no.~9-10, 859--886. \MR{3360394}

\bibitem{DPP2002}
{L.} De~Pascale and {A.} Pratelli, \emph{Regularity properties for {M}onge
  transport density and for solutions of some shape optimization problem},
  Calc. Var. Partial Differential Equations \textbf{14} (2002), no.~3,
  249--274. \MR{1899447}

\bibitem{Pr}
{L.} Prigozhin, \emph{Variational model of sandpile growth}, European J.\
  Appl.\ Math. \textbf{7} (1996), 225--235.

\bibitem{Sal}
{S.} Salsa, Partial differential equations in action, Universitext, vol.~99,
  Springer Milano, 2009.

\bibitem{San}
{F.} Santambrogio, \emph{Absolute continuity and summability of transport
  densities: simpler proofs and new estimates}, Calc. Var. Partial Differential
  Equations \textbf{36} (2009), no.~3, 343--354. \MR{2551134}

\end{thebibliography}
\end{document}